\documentclass[11pt]{amsart}%
\usepackage{amsmath}
\usepackage{amsfonts}
\usepackage{amssymb}
\usepackage{graphicx,delarray}%

\theoremstyle{plain}
\newtheorem{theorem}{Theorem}
\newtheorem*{theorem*}{Theorem}

\newtheorem{corollary}{Corollary}
\newtheorem{definition}{Definition}

\newtheorem{example}{Example}

\numberwithin{equation}{section}
\allowdisplaybreaks
\begin{document}

\title{Some results on Whitney numbers of Dowling lattices}

\author{Mourad Rahmani} \email{mrahmani@usthb.dz}
\address{USTHB, Faculty of Mathematics,
P. O. Box 32, El Alia,16111, Algiers, Algeria}

\begin{abstract}
In this paper, we study some properties of Whitney numbers of Dowling
lattices and related polynomials. We answer the following question: there is
relation between Stirling and Eulerian polynomials. Can we find a new
relation between Dowling polynomials and other polynomials generalizing
Eulerian polynomials? In addition, some congruences for the Dowling
numbers are given.
\end{abstract}
\maketitle

\section{Introduction}
In 1973, Dowling \cite{Dowling} introduced a class of geometric lattices called the
Dowling lattices. These lattices denoted $Q_{n}\left( G\right) $ are indexed
by a positive integer $n$ (rank) and a finite group $G$ of order $m\geq 1.$
The most important example of Dowling lattices is obtained by letting $G$ be
the trivial group $(e)$, then $Q_{n}\left( e\right) $ is the geometric
lattice of partitions $\Pi _{n+1}$ of the set $\{0,1,\ldots ,n\}.$

\bigskip Using M\"{o}bius function of a finite partially order set, Dowling gave
the characteristic polynomial of $Q_{n}\left( G\right) $
\[
P_{n}\left( \upsilon ;m\right) =m^{n}\left( \frac{\upsilon -1}{m}\right)
_{n},
\]%
where $\left( x\right) _{n}$ is the falling factorial defined by $\left( x\right) _{n}=x(x-1)\cdots(x-n+1)$, $\left( x\right) _{0}=1$.

It is well known that the Whitney numbers of the first kind $w_{m}\left(
n,k\right) $ are the coefficient of $\upsilon ^{k}$ of the characteristic
polynomial $P_{n}\left( \upsilon ;m\right) $ of $Q_{n}\left( G\right) $ and
Whitney numbers of the second kind $W_{m}\left( n,k\right) $ are the number
of elements of corank $k$ of $Q_{n}\left( G\right) .$ Dowling proved that
the Whitney numbers of Dowling lattices of both kinds satisfy the
orthogonality relations and also satisfy the following recursions%
\[
w_{m}\left( n,k\right) =\left( 1+m\left( n-1\right) \right) w_{m}\left(
n-1,k\right) +w_{m}\left( n-1,k-1\right)
\]%
and%
\[
W_{m}\left( n,k\right) =\left( 1+mk\right) W_{m}\left( n-1,k\right)
+W_{m}\left( n-1,k-1\right) .
\]

In 1996, Benoumhani \cite{benoumhani1,benoumhani3} established most properties (generating functions,
explicit formulas, recurrence relations, congruences, concavity) of Whitney
numbers of Dowling lattices. He also introduced two kinds
of polynomials \cite{benoumhani1,benoumhani2} related to Whitney numbers of Dowling lattices: the Dowling polynomials $D_{m}\left( n,x\right) $ and
Tanny-Dowling polynomials $\mathcal{F}_{m}\left( n,x\right) $.
The results reported in the present paper are complementary to those
obtained by Benoumhani and make points, especially in Eulerian-Dowling
polynomials. More precisely, the question which was asked by Benoumhani in
\cite{benoumhani1,benoumhani2} is: There is relation between Stirling and Eulerian polynomials. Can we find a
new relation between $\sum\limits_{k}W_{m}\left( n,k\right) x^{k}$ and
other polynomials generalizing Eulerian polynomials? The answer to the previous question is yes.

The present paper is organized as follows. We first introduce in section $2$, our notation and definitions. Then we present in section $3$ some
properties and some combinatorial identities related to the Dowling
polynomials and the Tanny-Dowling polynomials. The answer to the previous question
is in section $4$. Some congruences for
Dowling numbers are presented in section $5$. Finally, the $r$-Dowling numbers are also considered in section $6$.

\section{Definitions and notation}
In this section, we introduce some definitions and notation which are useful in the rest of the paper. The (signed) Stirling numbers of the first kind $s\left(  n,k\right)  $ are
the coefficients in the expansion%
\[
\left(  x\right)  _{n}=%
{\displaystyle\sum\limits_{k=0}^{n}}
s\left(  n,k\right)  x^{k}.
\]
Thus $s\left(  0,0\right)  =1,$ but $s\left(  n,0\right)  =0$ for $n\geq1,$
and it is also convenient to define $s\left(  n,k\right)  =0$ if $k<0$ or
$k>n.$ The recurrence%
\begin{equation}
s\left(  n+1,k\right)  =s\left(  n,k-1\right)  -n\text{ }s\left(  n,k\right)\label{st1}
\end{equation}
is well known and easy to see, and we also have the generating function%
\begin{equation}
\frac{1}{k!}\left(  \ln\left(  1+x\right)  \right)  ^{k}=%
{\displaystyle\sum\limits_{n\geq k}}
s\left(  n,k\right)  \frac{x^{n}}{n!}.\label{firstkind2}
\end{equation}

The Stirling numbers of the second kind, denoted $S\left(  n,k\right)  ,$ appear
as coefficients when converting powers to binomial coefficients%
\[
x^{n}=%
{\displaystyle\sum\limits_{k=0}^{n}}
k!S\left(  n,k\right)  \binom{x}{k}.
\]
They have a combinatorial interpretation involving set partitions.
Specifically,  $S\left(  n,k\right)  ,$ is the number of ways to partition a
set of $n$ elements into exactly $k$ nonempty subsets $\left(  0\leq k\leq
n\right)  $. The Stirling numbers of the second kind can be enumerated by the
following  recurrence relation%
\[
S\left(  n+1,k\right)  =kS\left(  n,k\right)  +S\left(  n,k-1\right)  ,
\]
or explicitly%
\[
S\left(  n,k\right)  =\frac{1}{k!}%
{\displaystyle\sum\limits_{j=1}^{k}}
\left(  -1\right)  ^{k-j}\binom{k}{j}j^{n}.
\]

The number of all partitions is the Bell number $\phi_{n},$ thus%
\[
\phi_{n}=%
{\displaystyle\sum\limits_{k=0}^{n}}
S\left(  n,k\right)  .
\]

The polynomials
\[
\phi_{n}\left(  x\right)  =%
{\displaystyle\sum\limits_{k=0}^{n}}
S\left(  n,k\right)  x^{k},
\]
are called Bell polynomials or exponential polynomials. The first few Bell polynomials are

\begin{align*}
\phi_{0}\left(  x\right)    & =1,\\
\phi_{1}\left(  x\right)    & =x,\\
\phi_{2}\left(  x\right)    & =x^{2}+x,\\
\phi_{3}\left(  x\right)    & =x^{3}+3x^{2}+x,\\
\phi_{4}\left(  x\right)    & =x^{4}+6x^{3}+7x^{2}+x.
\end{align*}

The exponential generating function for the polynomials $\phi_{n}\left(
x\right)  $ is%
\[%
{\displaystyle\sum\limits_{n\geq0}}
\phi_{n}\left(  x\right)  \frac{z^{n}}{n!}=\exp(x(e^{z}-1)).
\]

Now, if $\omega_{n}\left(  x\right)  $ and $\phi_{n}\left(  x\right)  $ are
ordinary and exponential generating functions of the sequence $k!S\left(
n,k\right)  ,$ then (cf. \cite{riordan})
\[
\omega_{n}\left(  x\right)  =%
{\displaystyle\int\limits_{0}^{+\infty}}
\phi_{n}\left(  \lambda x\right)  e^{-\lambda}d\lambda.
\]
The polynomials
\[
\omega_{n}\left(  x\right)  =%
{\displaystyle\sum\limits_{k=0}^{n}}
k!S\left(  n,k\right)  x^{k},
\]
are called geometric polynomials. The first few geometric polynomials are

\begin{align*}
\omega_{0}\left(  x\right)    & =1,\\
\omega_{1}\left(  x\right)    & =x,\\
\omega_{2}\left(  x\right)    & =2x^{2}+x,\\
\omega_{3}\left(  x\right)    & =6x^{3}+6x^{2}+x,\\
\omega_{4}\left(  x\right)    & =24x^{4}+36x^{3}+14x^{2}+x.
\end{align*}

The numbers $\omega_{n}\left(  1\right)  $ called ordered Bell numbers or
Fubini numbers, they count the number of ordered partitions of $\{1,2,\ldots,n\}$.

As Comtet in \cite[p. 244]{Comtet}, we define the Eulerian polynomials $A_{n}\left(
x\right)  $ by%
\begin{equation}
A_{n}\left(  x\right)  =\delta_{n,0}+%
{\displaystyle\sum\limits_{k=1}^{n}}
\genfrac{\langle}{\rangle}{0pt}{0}{n}{k-1}%
x^{k},\label{sm0}%
\end{equation}
where $%
\genfrac{\langle}{\rangle}{0pt}{0}{n}{k}%
$ are the Eulerian numbers. $%
\genfrac{\langle}{\rangle}{0pt}{0}{n}{k-1}%
$ is the number of permutations of length $n$ with exactly $k$ rises (i.e.,
the number of times it goes from a lower to a higher number, reading left to right).

The first few Eulerian polynomials are

\begin{align*}
A_{0}\left(  x\right)    & =1,\\
A_{1}\left(  x\right)    & =x,\\
A_{2}\left(  x\right)    & =x^{2}+x,\\
A_{3}\left(  x\right)    & =x^{3}+4x^{2}+x,\\
A_{4}\left(  x\right)    & =x^{4}+11x^{3}+11x^{2}+x.
\end{align*}
Using the Frobenius \cite{Frobenius} result
\begin{align}
A_{n}\left(  x\right)    & =\delta_{n,0}+x%
{\displaystyle\sum\limits_{k=1}^{n}}
k!S\left(  n,k\right)  \left(  x-1\right)  ^{n-k}\label{eqq1}\\
& =%
{\displaystyle\sum\limits_{k=0}^{n}}
k!S\left(  n+1,k+1\right)  \left(  x-1\right)  ^{n-k},\label{eqq2}%
\end{align}
we can easily establish the following connection between the Eulerian
polynomials and the geometric polynomials%
\[
A_{n}\left(  x\right)  =\delta_{n,0}+x\left(  x-1\right)  ^{n}\omega
_{n}\left(  \frac{1}{x-1}\right)  -x\left(  x-1\right)  ^{n},
\]
or
\begin{equation}
\omega_{n}\left(  x\right)  =\frac{x^{n+1}}{1+x}\left(  A_{n}\left(
\frac{1+x}{x}\right)  -\delta_{n,0}\right)  +1.\label{sm1}%
\end{equation}

Substituting (\ref{sm0}) in (\ref{sm1}) we get%
\[
\omega_{n}\left(  x\right)  =1+%
{\displaystyle\sum\limits_{k=0}^{n-1}}
\genfrac{\langle}{\rangle}{0pt}{0}{n}{k}%
\left(  1+x\right)  ^{k}x^{n-k},
\]
since $%
\genfrac{\langle}{\rangle}{0pt}{0}{n}{n}%
=\delta_{n,0}$\bigskip, we obtain the relationship between geometric polynomials and Eulerian numbers
\begin{equation}
\omega_{n}\left(  x\right)  =%
{\displaystyle\sum\limits_{k=0}^{n}}
\genfrac{\langle}{\rangle}{0pt}{0}{n}{k}%
\left(  1+x\right)  ^{k}x^{n-k}.\label{relation}
\end{equation}

It has been shown by Benoumhani that the first and second kind Whitney numbers
of Dowling lattices are \ defined respectively by%
\begin{align}
{\displaystyle\sum\limits_{n\geq k}}w_{m}\left(  n,k\right)  \frac{z^{n}}{n!}
&  =\frac{\left(  1+mz\right)  ^{-\frac{1}{m}}\left(  \ln\left(  1+mz\right)
\right)  ^{k}}{m^{k}k!},\label{A01}\\
{\displaystyle\sum\limits_{n\geq0}}W_{m}\left(  n,k\right)  \frac{z^{n}}{n!}
&  =\frac{e^{z}}{m^{k}k!}\left(  e^{mz}-1\right)  ^{k},\label{A02}%
\end{align}
or explicitly by%
\begin{align}
w_{m}\left(  n,k\right)    & =%
{\displaystyle\sum\limits_{i=0}^{n}}
\left(  -1\right)  ^{i-k}\binom{i}{k}m^{n-i}s\left(  n,i\right),  \label{d00}\\
W_{m}\left(  n,k\right)    & =%
{\displaystyle\sum\limits_{i=k}^{n}}
\binom{n}{i}m^{i-k}S\left(  i,k\right)  \label{d01}\\
& =\frac{1}{m^{k}k!}%
{\displaystyle\sum\limits_{i=0}^{k}}
\binom{k}{i}\left(  -1\right)  ^{k-i}\left(  mi+1\right)  ^{n}.\label{d02}%
\end{align}

For more details on Whitney numbers of Dowling lattices see \cite{benoumhani1,Dowling}.

The Dowling polynomials and Tanny-Dowling polynomials were evidently first
introduced by Benoumhani \cite{benoumhani1,benoumhani2}. They are usually defined in the following
way:%
\begin{align}
D_{m}\left(  n,x\right)   &  =%
{\displaystyle\sum\limits_{k=0}^{n}}
W_{m}\left(  n,k\right)  x^{k}\label{b02}\\
\mathcal{F}_{m}\left(  n,x\right)   &  =%
{\displaystyle\sum\limits_{k=0}^{n}}
k!W_{m}\left(  n,k\right)  x^{k}.\label{b03}%
\end{align}
It is not difficult to see that
\[
\mathcal{F}_{m}\left(  n,x\right)  =%
{\displaystyle\int\limits_{0}^{+\infty}}
D_{m}\left(  n,\lambda x\right)  e^{-\lambda}d\lambda.
\]

\section{Some proprieties of the Dowling polynomials}

\begin{theorem}
For $m\geq1$, the Whitney numbers of the second kind $W_{m}\left(  n,k\right)
$ satisfy the recursion
\begin{equation}
W_{m+1}\left(  n,k\right)  =\frac{1}{\left(  m+1\right)  ^{k}m^{n-k}}%
{\displaystyle\sum\limits_{j=0}^{n}}
\left(  -1\right)  ^{n-j}\binom{n}{j}\left(  m+1\right)  ^{j}W_{m}\left(
j,k\right)  ,\label{rec1}
\end{equation}
with $W_{1}\left(  n,k\right)  =S\left(  n+1,k+1\right)  .$
\end{theorem}

\begin{proof}
Expression (\ref{d02}) may be rewritten as%
\begin{align*}
W_{m+1}\left(  n,k\right)   &  =\frac{\left(  m+1\right)  ^{n-k}}{k!}%
{\displaystyle\sum\limits_{i=0}^{k}}
\binom{k}{i}\left(  -1\right)  ^{k-i}\left(  i+\frac{1}{m+1}\right)  ^{n}\\
&  =\frac{\left(  m+1\right)  ^{n}}{\left(  m+1\right)  ^{k}k!}%
{\displaystyle\sum\limits_{i=0}^{k}}
\binom{k}{i}\left(  -1\right)  ^{k-i}\left(  i+\frac{1}{m}-\frac{1}{m\left(
m+1\right)  }\right)  ^{n}\\
&  =\frac{\left(  m+1\right)  ^{n}}{\left(  m+1\right)  ^{k}k!}%
{\displaystyle\sum\limits_{i=0}^{k}}
\binom{k}{i}\left(  -1\right)  ^{k-i}%
{\displaystyle\sum\limits_{j=0}^{n}}
\binom{n}{j}\left(  i+\frac{1}{m}\right)  ^{n-j}\left(  -\frac{1}{m\left(
m+1\right)  }\right)  ^{j}\\
&  =\left(  m+1\right)  ^{n}%
{\displaystyle\sum\limits_{j=0}^{n}}
\binom{n}{j}\frac{m^{k}}{\left(  m+1\right)  ^{k}m^{n-j}}\left(  -\frac
{1}{m\left(  m+1\right)  }\right)  ^{j}\frac{1}{m^{k}k!}%
{\displaystyle\sum\limits_{i=0}^{k}}
\binom{k}{i}\left(  -1\right)  ^{k-i}\left(  mi+1\right)  ^{n-j}\\
&  =\frac{1}{\left(  m+1\right)  ^{k}m^{n-k}}%
{\displaystyle\sum\limits_{j=0}^{n}}
\left(  -1\right)  ^{j}\binom{n}{j}\left(  m+1\right)  ^{n-j}W_{m}\left(
n-j,k\right) ,
\end{align*}
which is the required expression (\ref{rec1}).
\end{proof}

\begin{theorem}
For $m\geq1$, the Dowling polynomials $D_{m}\left(  n,x\right)  $ satisfy the
recursion%
\begin{equation}
D_{m+1}\left(  n,x\right)  =\frac{1}{m^{n}}%
{\displaystyle\sum\limits_{j=0}^{n}}
\left(  -1\right)  ^{n-j}\binom{n}{j}\left(  m+1\right)  ^{j}D_{m}\left(
j,\frac{mx}{m+1}\right)  ,\label{recc1}
\end{equation}
with $D_{1}\left(  n,x\right)  =x^{-1}\phi_{n+1}\left(  x\right)  .$
\end{theorem}

\begin{proof}
By using (\ref{b02}) and (\ref{rec1}), we obtain%
\[
D_{m+1}\left(  n,x\right)  =\frac{1}{m^{n}}%
{\displaystyle\sum\limits_{j=0}^{n}}
\left(  -1\right)  ^{n-j}\binom{n}{j}\left(  m+1\right)  ^{j}%
{\displaystyle\sum\limits_{k=0}^{n}}
W_{m}\left(  j,k\right)  \left(  \frac{m}{m+1}x\right)  ^{k},
\]
we arrive at the desired result.
\end{proof}

\begin{theorem}\label{th3}
For $m\geq1$, we have%
\begin{equation}
\bigskip\mathcal{F}_{m+1}\left(  n,x\right)  =\frac{1}{m^{n}}%
{\displaystyle\sum\limits_{j=0}^{n}}
\left(  -1\right)  ^{n-j}\binom{n}{j}\left(  m+1\right)  ^{j}\mathcal{F}%
_{m}\left(  j,\frac{m}{m+1}x\right)  ,\label{sim0}
\end{equation}
with $\mathcal{F}_{1}\left(  n,x\right)  =\left(  1+\frac{1}{x}\right)
\omega_{n}\left(  x\right)  -\frac{\delta_{n,0}}{x}.$
\end{theorem}

\begin{proof}
Combining (\ref{b03}) and (\ref{rec1}), we easily arrive at the desired result.
\end{proof}

\begin{theorem}
For $m\geq1$, the Dowling polynomials $D_{m}\left(  n,x\right)  $ satisfy
\begin{equation}
D_{m}\left(  n,x\right)  =\sum_{i=0}^{n}\binom{n}{i}m^{i}\phi_{i}\left(
\frac{x}{m}\right)  . \label{resulta1}%
\end{equation}

\end{theorem}

\begin{proof}
By using (\ref{b02}) and (\ref{d01}), we get%
\[
D_{m}\left(  n,x\right)  =%
{\displaystyle\sum\limits_{i=0}^{n}}
\binom{n}{i}m^{i}%
{\displaystyle\sum\limits_{k=0}^{n}}
S\left(  i,k\right)  \left(  \frac{x}{m}\right)  ^{k}.%
\]

\end{proof}

We note that the identity (\ref{resulta1}) can be viewed as a binomial
transform. Given a sequence $\alpha_{k},$ its binomial transform $\beta_{k}$
is the sequence defined by
\[
\beta_{n}=%
{\displaystyle\sum\limits_{k=0}^{n}}
\binom{n}{k}\alpha_{k},\text{ with inversion }\alpha_{n}=%
{\displaystyle\sum\limits_{k=0}^{n}}
\left(  -1\right)  ^{n-k}\binom{n}{k}\beta_{k}.
\]

From this observation we obtain
\begin{corollary}%
\begin{equation}
\phi_{n}\left(  \frac{x}{m}\right)  =\frac{1}{m^{n}}\sum_{i=0}^{n}\binom{n}%
{i}(-1)^{n-i}D_{m}\left(  i,x\right).  \label{resulta2}%
\end{equation}

\end{corollary}

By substituting $m=1$ in (\ref{resulta1}) and (\ref{resulta2}), we obtain the well-known results.

\begin{align}
x\phi_{n}\left(  x\right)   &  =\sum_{i=0}^{n}\binom{n}{i}(-1)^{n-i}\phi
_{i+1}\left(  x\right),\label{c01} \\
\phi_{n+1}\left(  x\right)   &  =x\sum_{i=0}^{n}\binom{n}{i}\phi_{i}\left(
x\right) \label{c02} .
\end{align}

As the result of Chen \cite{chen} for the binomial transform, we have%
\begin{equation}
\sum_{k=0}^{l}\binom{l}{k}\binom{n+k}{s}\alpha_{n+k-s}=\sum_{k=0}^{n}\binom
{n}{k}\binom{l+k}{s}\left(  -1\right)  ^{n-k}\beta_{l+k-s}.\label{sim1}%
\end{equation}
Substituting $\alpha_{k}:=m^{k}\phi_{k}\left(  \frac{x}{m}\right)  ,\beta
_{k}:=D_{m}\left(  \,k,x\right)  $ and $l=s=n$ into (\ref{sim1}), we get a
curious identity of Simons type (see \cite{boya}) which has the
interesting property that the binomial coefficient on both sides are the same%
\begin{equation}
\sum_{k=0}^{n}\binom{n}{k}\binom{n+k}{k}m^{k}\phi_{k}\left(  \frac{x}%
{m}\right)  =\sum_{k=0}^{n}\binom{n}{k}\binom{n+k}{k}\left(  -1\right)
^{n-k}D_{m}\left(  \,k,x\right)  ,\label{f01}
\end{equation}
and for $m=1,$ we have a curious identity for Bell polynomials%
\begin{equation}
\sum_{k=0}^{n}\binom{n}{k}\binom{n+k}{k}x\phi_{k}\left(  x\right)  =\sum
_{k=0}^{n}\binom{n}{k}\binom{n+k}{k}\left(  -1\right)  ^{n-k}\phi_{k+1}\left(
x\right)\label{f02}
\end{equation}

Now, setting $m=1,x:=2x$ in (\ref{recc1}) and setting $m=2,x:=2x$ in (\ref{resulta1}), we get another
curious identity for Bell polynomials%

\begin{corollary}
The following formula holds true
\[
\sum_{k=0}^{n}\binom{n}{k}2^{k}x\phi_{k}\left(  x\right)  =\sum_{k=0}%
^{n}\binom{n}{k}2^{k}\left(  -1\right)  ^{n-k}\phi_{k+1}\left(  x\right).
\]
\end{corollary}

Similarly, we obtain

\begin{theorem}
For $m\geq1,$ we have

\begin{align}
\mathcal{F}_{m}\left(  n,x\right)   &  =%
{\displaystyle\sum\limits_{i=0}^{n}}
\binom{n}{i}m^{i}\omega_{i}\left(  \frac{x}{m}\right),  \label{f1}\\
\omega_{n}\left(  \frac{x}{m}\right)   &  =\frac{1}{m^{n}}%
{\displaystyle\sum\limits_{i=0}^{n}}
\binom{n}{i}\left(  -1\right)  ^{n-i}\mathcal{F}_{m}\left(  i,x\right),
\label{f2}%
\end{align}

\begin{align}
x\omega_{n}\left(  x\right)    & =\left(  -1\right)  ^{n+1}+%
{\displaystyle\sum\limits_{i=0}^{n}}
\binom{n}{i}\left(  -1\right)  ^{n-i}\left(  x+1\right)  \omega_{i}\left(
x\right),  \label{f3}\\
\left(  1+x\right)  \omega_{n}\left(  x\right)    & =\delta_{n,0}+%
{\displaystyle\sum\limits_{i=0}^{n}}
\binom{n}{i}x\omega_{i}\left(  x\right),  \label{f4}%
\end{align}

\begin{align}
\sum_{k=0}^{n}\binom{n}{k}\binom{l+k}{s}\left(  -1\right)  ^{n-k}%
\mathcal{F}_{m}\left(  l+k-s,x\right)    & =\sum_{k=0}^{l}\binom{l}{k}%
\binom{n+k}{s}m^{n+k-s}\omega_{n+k-s}\left(  \frac{x}{m}\right),  \label{f5}\\
\sum_{k=0}^{n}\binom{n}{k}\binom{n+k}{k}x\omega_{k}\left(  x\right)    &
=1+\sum_{k=0}^{n}\binom{n}{k}\binom{n+k}{k}\left(  -1\right)  ^{n-k}%
\omega_{k}\left(  x\right),  \label{f6}%
\end{align}

\begin{equation}
\sum_{k=0}^{n}\binom{n}{k}2^{k}x\omega_{k}\left(  x\right)  =\left(  -1\right)
^{n}+\sum_{k=0}^{n}\binom{n}{k}2^{k}\left(  -1\right)  ^{n-k}\left(
1+x\right)  \omega_{k}\left(  x\right).  \label{f7}%
\end{equation}
\end{theorem}
\begin{proof}
The identity (\ref{f1}) can be found in \cite{benoumhani2}, here we give a simple proof. Combining (\ref{b03}) and (\ref{d01}), we get%
\[
\sum_{k=0}^{n}k!W_{m}\left(  n,k\right)  x^{k}=%
{\displaystyle\sum\limits_{i=0}^{n}}
\binom{n}{i}m^{i}\sum_{k=0}^{n}k!S\left(  i,k\right)  \left(  \frac{x}%
{m}\right)  ^{k}.
\]
The relation (\ref{f2}) is the inverse binomial transform of identity (\ref{f1}).
As a special case, we get (\ref{f3}) and (\ref{f4}) by using Theorem \ref{th3} and
setting $m=1$ in (\ref{f1}) and (\ref{f2}).

By substituting $\alpha_{k}:=m^{k}\omega_{k}\left(  \frac{x}{m}\right)  $ and
$\beta_{k}:=\mathcal{F}_{m}\left(  \,k,x\right)  $ in (\ref{sim1}), we get
(\ref{f5}). The relation (\ref{f6}) is a special case, by setting $m=1$ and
$n=s=l$ in (\ref{f5}).

Finally, setting $m=1$, $x:=2x$ in (\ref{sim0}) and setting $m=2$, $x:=2x$ in (\ref{f1}), we get
(\ref{f7}). This completes the proof.

\end{proof}

The Hankel transform of a sequence $\alpha_{n}$ is the sequence of Hankel
determinants $H_{n}\left(  \alpha_{n}\right)  $, where $H_{n}\left(
\alpha_{n}\right)  =\det\left(  \alpha_{i+j}\right)  _{0\leq i,j\leq n}$. It is well known that the Hankel transform of a sequences $\alpha_{n}$ and
$\beta_{n}$ are equal (see \cite{layman}).

In 2000, Suter \cite{suter} proved that
$H_{n}\left(  D_{m}\left(  n,1\right)  \right)  =m^{\binom{n+1}{2}}%
{\displaystyle\prod\limits_{k=1}^{n}}
k!,$ we shall give the following generalization

\begin{corollary}%
\[
H_{n}\left(  D_{m}\left(  n,x\right)  \right)  =\left(  xm\right)
^{\binom{n+1}{2}}%
{\displaystyle\prod\limits_{k=1}^{n}}
k!.
\]

\end{corollary}

\begin{proof}
Using the fact that $H_{n}\left(  D_{m}\left(  n,x\right)  \right)
=H_{n}\left(  m^{n}\phi_{n}\left(  \frac{x}{m}\right)  \right)  $ and
$H_{n}\left(  \phi_{n}\left(  x\right)  \right)  =x^{\binom{n+1}{2}}%
{\displaystyle\prod\limits_{k=1}^{n}}
k!$ (cf. \cite{Radoux}).
\end{proof}

\section{The Eulerian-Dowling polynomials}
In this section, we define the Eulerian-Dowling polynomials and we derive
some elementary properties. According to (\ref{eqq2}), the following definition provides a natural
generalization of Eulerian polynomials.

\begin{definition}
The Eulerian-Dowling polynomials $\mathcal{A}_{m}\left(  n,x\right)  $ are
defined by%
\begin{align}
\mathcal{A}_{m}\left(  n,x\right)   &  =%
{\displaystyle\sum\limits_{i=0}^{n}}
i!W_{m}\left(  n,i\right)  \left(  x-1\right)  ^{n-i}\label{eul1}\\
&  =\left(  x-1\right)  ^{n}\mathcal{F}_{m}\left(  n,\frac{1}{x-1}\right)
\label{eul2}%
\end{align}

\end{definition}

From the above definition, we can rewrite $\mathcal{A}_{m}\left(  n,x\right)  $
as
\begin{align*}
\mathcal{A}_{m}\left(  n,x\right)   &  =%
{\displaystyle\sum\limits_{i=0}^{n}}
{\displaystyle\sum\limits_{k=0}^{n-i}}
\binom{n-i}{k}i!W_{m}\left(  n,i\right)  \left(  -1\right)  ^{n-i-k}x^{k}.\\
&  =%
{\displaystyle\sum\limits_{k=0}^{n}}
\left(
{\displaystyle\sum\limits_{i=0}^{n}}
\left(  -1\right)  ^{n-i-k}\binom{n-i}{k}i!W_{m}\left(  n,i\right)  \right)
x^{k},
\end{align*}
\bigskip

Now, we define the Eulerian-Dowling numbers $a_{m}\left(  n,k\right)  $
by
\begin{equation}
a_{m}\left(  n,k\right)  =%
{\displaystyle\sum\limits_{i=0}^{n}}
\left(  -1\right)  ^{n-i-k}i!\binom{n-i}{k}W_{m}\left(  n,i\right)
.\label{eul0}%
\end{equation}
For $m=1,$ we have%
\[
a_{1}\left(  n,k\right)  =\delta_{n,0}+%
\genfrac{\langle}{\rangle}{0pt}{}{n}{k-1}%
.
\]

The following elementary properties of the Eulerian-Dowling polynomials
are given

\begin{theorem}
The exponential generating function for $\mathcal{A}_{m}\left(  n,x\right)  $
is%
\[%
{\displaystyle\sum\limits_{n\geq0}}
\mathcal{A}_{m}\left(  n,x\right)  \frac{z^{n}}{n!}=\frac{m\left(  x-1\right)
e^{\left(  x-1\right)  z}}{m\left(  x-1\right)  +1-e^{m\left(  x-1\right)  z}%
}.
\]

\end{theorem}

\begin{proof}
From (\ref{eul1}$)$ and (\ref{A02})%
\begin{align*}%
{\displaystyle\sum\limits_{n\geq0}}
\mathcal{A}_{m}\left(  n,x\right)  \frac{z^{n}}{n!}  &  =%
{\displaystyle\sum\limits_{i\geq0}}
i!\frac{1}{\left(  x-1\right)  ^{i}}%
{\displaystyle\sum\limits_{n\geq i}}
W_{m}\left(  n,i\right)  \frac{\left(  z\left(  x-1\right)  \right)  ^{n}}%
{n!}\\
&  =\exp\left(  z\left(  x-1\right)  \right)
{\displaystyle\sum\limits_{i\geq0}}
\left(  \frac{\exp\left(  zm\left(  x-1\right)  \right)  -1}{m\left(
x-1\right)  }\right)  ^{i},
\end{align*}
we arrive at the desired result.
\end{proof}

\bigskip

In \cite{benoumhani1,benoumhani2}, Benoumhani asked for the analogue of (\ref{relation}) for $\mathcal{F}_{m}\left(
n,x\right)  $. $\ $The answer to the previous question is given in the
following theorem

\begin{theorem}%
\begin{equation}
\mathcal{F}_{m}\left(  n,x\right)  =%
{\displaystyle\sum\limits_{k=0}^{n}}
a_{m}\left(  n,k\right)  \left(  1+x\right)  ^{k}x^{n-k}. \label{eul3}%
\end{equation}

\end{theorem}

\begin{proof}
From (\ref{eul2}), we can write $\mathcal{F}_{m}\left(  n,x\right)  $ as%
\begin{align*}
\mathcal{F}_{m}\left(  n,x\right)   &  =x^{n}\mathcal{A}_{m}\left(
n,\frac{1+x}{x}\right) \\
&  =x^{n}%
{\displaystyle\sum\limits_{k=0}^{n}}
a_{m}\left(  n,k\right)  \left(  \frac{1+x}{x}\right)  ^{k},
\end{align*}

which completes the proof.
\end{proof}

As a special case, we have the well known result

\begin{corollary}%
\[
\omega_{n}\left(  1\right)  =%
{\displaystyle\sum\limits_{k=0}^{n}}
\genfrac{\langle}{\rangle}{0pt}{}{n}{k}%
2^{k}.
\]

\end{corollary}

\begin{proof}
By setting $m=1,$ $x=1$ in (\ref{eul3}) and using Theorem \ref{th3}, we get%
\[
2\omega_{n}\left(  1\right)  -\delta_{n,0}=%
{\displaystyle\sum\limits_{k=0}^{n}}
\left(  \delta_{n,0}+%
\genfrac{\langle}{\rangle}{0pt}{}{n}{k-1}%
\right)  2^{k},
\]
from which it follows that
\[
\omega_{n}\left(  1\right)  =\delta_{n,0}+%
{\displaystyle\sum\limits_{k=0}^{n-1}}
\genfrac{\langle}{\rangle}{0pt}{}{n}{k}%
2^{k},
\]
since $%
\genfrac{\langle}{\rangle}{0pt}{}{n}{n}%
=\delta_{n,0},$ we obtain the result.
\end{proof}

\section{Congruences for Dowling numbers}
By using the Gessel method \cite{Gessel}, we shall give some congruences for the Dowling numbers. We consider the polynomials $R_{n,k}^{\left(  m\right)  }\left(  t\right)  $
for fixed $m$, defined by the exponential generating function%
\begin{equation}%
{\displaystyle\sum\limits_{n\geq k}}
R_{n,k}^{\left(  m\right)  }\left(  t\right)  \frac{z^{n}}{n!}=e^{-tz}\left(
1+mz\right)  ^{-\frac{1}{m}}\frac{\left(  \ln(1+mz)\right)  ^{k}}{m^{k}%
k!}.\label{rel1}%
\end{equation}

\begin{theorem}
The following explicit representation formula holds true
\begin{equation}
R_{n,k}^{\left(  m\right)  }\left(  t\right)  =%
{\displaystyle\sum\limits_{j=0}^{n}}
\left(  -1\right)  ^{j}\binom{n}{j}w_{m}\left(  n-j,k\right)  t^{j}.\label{exp1}
\end{equation}
Here $w_{m}\left(  n,k\right)  $ are the Whitney numbers of the first kind.
\end{theorem}

\begin{proof}
From the generating function (\ref{A01}) we have%
\begin{align*}%
{\displaystyle\sum\limits_{n\geq k}}
R_{n,k}^{\left(  m\right)  }\left(  t\right)  \frac{z^{n}}{n!}  &  =%
{\displaystyle\sum\limits_{n\geq0}}
\left(  -1\right)  ^{n}t^{n}\frac{z^{n}}{n!}%
{\displaystyle\sum\limits_{n\geq k}}
w_{m}\left(  n,k\right)  \frac{z^{n}}{n!}\\
&  =%
{\displaystyle\sum\limits_{n\geq0}}
\frac{z^{n}}{n!}%
{\displaystyle\sum\limits_{j=0}^{n}}
\left(  -1\right)  ^{j}\binom{n}{j}w\left(  n-j,k\right)  t^{j}.
\end{align*}
Equating the coefficients of $\frac{z^{n}}{n!}$ we get the result.
\end{proof}

\begin{theorem}
The double generating function for $R_{n,k}^{\left(  m\right)  }\left(
t\right)  $ is%
\begin{equation}%
{\displaystyle\sum\limits_{n\geq0,k\geq0}}
R_{n,k}^{\left(  m\right)  }\frac{z^{n}}{n!}u^{k}=e^{-tz}\left(  1+mz\right)
^{\frac{u-1}{m}}.\label{rel2}%
\end{equation}

\end{theorem}

\begin{proof}
From (\ref{rel1})%
\begin{align*}%
{\displaystyle\sum\limits_{n\geq k}}
R_{n,k}^{\left(  m\right)  }\left(  t\right)  \frac{z^{n}}{n!}%
{\displaystyle\sum\limits_{k}}
u^{k}  & =\left(  1+mz\right)  ^{-\frac{1}{m}}e^{-tz}%
{\displaystyle\sum\limits_{k}}
\left(  \frac{\ln(1+mz)u}{m}\right)  ^{k}\frac{1}{k!}\\
& =\left(  1+mz\right)  ^{-\frac{1}{m}}e^{-z}\exp(\ln(1+mz)^{\frac{u}{m}}).
\end{align*}

\end{proof}

\begin{theorem}
The $R_{n,k}^{\left(  m\right)  }\left(  t\right)  $ satisfy the following
recurrence relation%
\begin{equation}
R_{n+1,k}^{\left(  m\right)  }\left(  t\right)  =R_{n,k-1}^{\left(  m\right)
}\left(  t\right)  -(\left(  1+t\right)  +mn)R_{n,k}^{\left(  m\right)
}\left(  t\right)  -mntR_{n-1,k}^{\left(  m\right)  },\label{rel3}%
\end{equation}
with initial conditions $R_{0,0}^{\left(  m\right)  }\left(  t\right)  =1$ and
$R_{n,k}^{\left(  m\right)  }\left(  t\right)  =0$ if $k>n$ or $k<0.$
\end{theorem}

\begin{proof}
Let \ $R\left(  u,z\right)  $ $\ $\ be the double generating function
(\ref{rel2}). Then by differentiation with respect to $z$ we obtain%
\[
\left(  1+mz\right)  \frac{d}{dz}R\left(  u,z\right)  =\left(  u-1-t\left(
1+mz\right)  \right)  R\left(  u,z\right),
\]
or equivalently%
\[%
{\displaystyle\sum\limits_{n,k}}
\left(  R_{n+1,k}^{\left(  m\right)  }+mnR_{n,k}^{\left(  m\right)  }\right)
\frac{z^{n}}{n!}u^{k}=%
{\displaystyle\sum\limits_{n,k}}
\left(  R_{n,k-1}^{\left(  m\right)  }-\left(  1+t\right)  R_{n,k}^{\left(
m\right)  }-tmnR_{n-1,k}^{\left(  m\right)  }\right)  \frac{z^{n}}{n!}u^{k}.
\]

Comparing the coefficients of $\frac{z^{n}}{n!}u^{k}$ on both sides of the above equation, we arrive at the desired
result.
\end{proof}

\bigskip

Taking $t=1$ in (\ref{rel3}) and a little computation gives the following table of values : Table \ref{tab1}
\begin{table}
\caption{$R_{n,k}^{\left(  m\right)  }\left(1\right)  $\label{tab1}}
\medskip
\begin{center}
\begin{tabular}
[c]{l|cccccc}%
$n\left\backslash k\right.  $ & $0$ & $1$ & $2$ & $3$ & $4$ \\\hline
$0$ & $1$ &  &  &  &  & \\
$1$ & $-2$ & $1$ &  &  &  & \\
$2$ & $m+4$ & $-m-4$ & $1$ &  &  & \\
$3$ & $-2m^{2}-6m-8$ & $2m^{2}+9m+12$ & $-3m-6$ & $1$ &  & \\
$4$ & $6m^{3}+19m^{2}+24m+16$ & $-6m^{3}-30m^{2}-48m-32$ & $11m^{2}+30m+24$ &
$-6m-8$ & $1$ & \\
\end{tabular}
\end{center}
\end{table}

\begin{theorem}
\begin{align}%
{\displaystyle\sum\limits_{k=0}^{n}}
R_{n,k}^{\left(  m\right)  }\left(  t\right)  D_{m}\left(  i+k,t\right)    & =t^{n}n!%
{\displaystyle\sum\limits_{j=0}^{i}}
m^{i-j}\binom{i}{j}S\left(  i-j,n\right)  D_{m}\left(  j,t\right)
\label{cong4}\\
& =\left\{
\begin{tabular}
[c]{lr}%
$n!m^{n}t^{n},$ & $\text{\ }i=n$\\
$0,$ & \multicolumn{1}{l}{$\text{ }0\leq i<n$}%
\end{tabular}
\right.  .\nonumber
\end{align}
where $R_{n,k}^{\left(  m\right)  }\left(  t\right)$ is defined in (\ref{exp1}).
\end{theorem}

\begin{proof}
Let $f\left(  x\right)  $ be the generating function for the Dowling
polynomials, so that%
\[
f(x)=\exp(x+\frac{t}{m}\left(  e^{mx}-1\right)  ).
\]
Then $f\left(  x\right)  $ satisfies the functional equation%
\begin{equation}
f(x+y)=f\left(  x\right)  \exp(y+\frac{t}{m}\left(  e^{my}-1)e^{mx}\right)
).\label{cong1}%
\end{equation}
Using Taylor's theorem, we have
\[
f(x+y)=\sum_{k\geq0}f^{\left(  k\right)  }\left(  x\right)  \frac{y^{k}}{k!},
\]
where $f^{\left(  k\right)  }\left(  x\right)  =\frac{d^{k}}{dx^{k}}f\left(
x\right)  ,$ it follows that
\begin{equation}
f^{\left(  k\right)  }\left(  x\right)  =%
{\displaystyle\sum\limits_{i\geq0}}
D_{m}\left(  i+k,t\right)  \frac{x^{i}}{i!}.\label{cong2}%
\end{equation}
Now, set $z=\frac{\exp(my)-1}{m}$ in (\ref{cong1}), we get%
\[
\sum_{k\geq0}f^{\left(  k\right)  }\left(  x\right)  \frac{\left[
\ln(1+mz)\right]  ^{k}}{m^{k}k!}=f\left(  x\right)  \left(  1+mz\right)
^{\frac{1}{m}}\exp(tze^{mx}).
\]
Multiplying both sides by $e^{-tz}\left(  1+mz\right)  ^{-\frac{1}{m}}$ we
obtain
\[
f\left(  x\right)  \exp(tz\left(  e^{mx}-1\right)  )=\sum_{k\geq0}f^{\left(
k\right)  }\left(  x\right)  e^{-tz}\left(  1+mz\right)  ^{-\frac{1}{m}}%
\frac{\left[  \ln(1+mz)\right]  ^{k}}{m^{k}k!},
\]
and
\[
f\left(  x\right)  \exp(tz\left(  e^{mx}-1\right)  )=%
{\displaystyle\sum\limits_{n\geq0}}
f\left(  x\right)  \frac{t^{n}z^{n}}{n!}\left(  e^{mx}-1\right)  ^{n}.
\]
Since
\[
e^{-tz}\left(  1+mz\right)  ^{-\frac{1}{m}}\frac{\left[  \ln(1+mz)\right]
^{k}}{m^{k}k!}=%
{\displaystyle\sum\limits_{n\geq k}}
R_{n,k}^{\left(  m\right)  }\left(  t\right)  \frac{z^{n}}{n!},
\]
and
\begin{equation}
\frac{1}{n!}\left(  e^{mx}-1\right)  ^{n}=%
{\displaystyle\sum\limits_{i\geq n}}
m^{i}S\left(  i,n\right)  \frac{x^{i}}{i!}.\label{eng3}%
\end{equation}
It follows from (\ref{cong2}) that
\begin{align*}
\sum_{k\geq0}f^{\left(  k\right)  }\left(  x\right)  \left(  1+mz\right)
^{-\frac{1}{m}}e^{-tz}\frac{\left[  \ln(1+mz)\right]  ^{k}}{m^{k}k!}  &
=\sum_{k\geq0}f^{\left(  k\right)  }\left(  x\right)
{\displaystyle\sum\limits_{n\geq0}}
R_{n,k}^{\left(  m\right)  }\left(  t\right)  \frac{z^{n}}{n!}\\
& =\sum_{n\geq0}\sum_{i\geq0}\frac{x^{i}}{i!}\frac{z^{n}}{n!}%
{\displaystyle\sum\limits_{k=0}^{n}}
R_{n,k}^{\left(  m\right)  }\left(  t\right)  D_{m}\left(  i+k,t\right)  ,
\end{align*}

and by (\ref{eng3}), we get
\begin{align*}%
{\displaystyle\sum\limits_{n\geq0}}
t^{n}z^{n}f\left(  x\right)  \frac{1}{n!}\left(  e^{mx}-1\right)  ^{n}  & =%
{\displaystyle\sum\limits_{n\geq0}}
t^{n}z^{n}\sum_{n\geq0}D_{m}\left(  n,t\right)  \frac{x^{n}}{n!}%
{\displaystyle\sum\limits_{i\geq n}}
m^{i}S\left(  i,n\right)  \frac{x^{i}}{i!}\\
& =%
{\displaystyle\sum\limits_{n\geq0}}
{\displaystyle\sum\limits_{i\geq0}}
\frac{x^{i}}{i!}\frac{z^{n}}{n!}t^{n}n!%
{\displaystyle\sum\limits_{j=0}^{i}}
m^{i-j}\binom{i}{j}S\left(  i-j,n\right)  D_{m}\left(  j,t\right).
\end{align*}

Equating coefficients of $\frac{x^{i}}{i!}\frac{z^{n}}{n!}$, we get the results.
\end{proof}

It is clear from $\left(  \ref{cong4}\right)  $ that the right-hand side is
divisible by $n!$.

\begin{corollary}
Let $n,i$ be non-negative integers with $i\leq n$, we have%
\begin{equation}%
{\displaystyle\sum\limits_{k=0}^{n}}
R_{n,k}^{\left(  m\right)  }\left(  t\right)  D_{m}\left(  i+k,t\right)  \equiv0\text{ }\left(
\operatorname{mod}n!\right)  .\label{cong5}%
\end{equation}

\end{corollary}

Let us give a short list of these congruences by taking $t=1$ in (\ref{cong5})
and using the Table \ref{tab1}.

\begin{gather*}
mD_{m}\left(  i\right)  +mD_{m}\left(  i+1\right)  +D_{m}\left(  i+2\right)
\equiv0\text{ }\left(  \operatorname{mod}2\right)  ,\\
\left(  4m^{2}-2\right)  D_{m}\left(  i\right)  +\left(  2m^{2}+3m\right)
D_{m}\left(  i+1\right)  +3mD_{m}\left(  i+2\right)  +D_{m}\left(  i+3\right)
\equiv0\text{ }\left(  \operatorname{mod}6\right).
\end{gather*}
\section{$r$-Dowling polynomials}

In $1984$, Broder \cite{Broder} generalized the Stirling numbers
of the second kind to the so-called $r-$Stirling numbers of the second kind $%
\genfrac{\{}{\}}{0pt}{0}{n}{k}%
_{r}$ as follows: is the number of
partitions of $\left\{  1,2,\ldots,n\right\}  $ into exactly $k$ nonempty,
disjoint subsets, such that the first $r$ elements are in distinct subsets. They may be
 defined recursively as follows
\begin{equation}
\begin{tabular}
[c]{lll}%
$
\genfrac{\{}{\}}{0pt}{0}{n}{k}%
_{r}=0,$ &  & $n<r,$\\
$
\genfrac{\{}{\}}{0pt}{0}{n}{k}%
_{r}=\delta_{k,r},$ &  & $\text{ }n=r,$\\
$%
\genfrac{\{}{\}}{0pt}{0}{n}{k}%
_{r}=k%
\genfrac{\{}{\}}{0pt}{0}{n-1}{k}%
_{r}+%
\genfrac{\{}{\}}{0pt}{0}{n-1}{k-1}%
_{r},$ &  & $n>r,$\label{tab}
\end{tabular}
\end{equation}
where $\delta_{k,r}$ is the Kronecker symbol.

The $r$-Whitney numbers of both kinds have appeared in \cite{mezo} as a common
generalization of Whitney numbers and $r$-Stirling numbers. Recently, Choen
and Jung \cite{cheon} have used these numbers to extend earlier results of Benoumhani.
They defined the $r$-Dowling polynomials by means of
\begin{equation}
D_{m,r}\left(  n,x\right)  =%
{\displaystyle\sum\limits_{k=0}^{n}}
W_{m,r}\left(  n,k\right)  x^{k},\label{rdow}%
\end{equation}
where $W_{m,r}\left(  n,k\right)  $ is the $r$-Whitney numbers of the second
kind of the Dowling lattices $Q_{n}\left(  G\right)  $ defined by%
\begin{equation}
W_{m,r}\left(  n,k\right)  =%
{\displaystyle\sum\limits_{j=k}^{n}}
\binom{n}{j}m^{j-k}\left(  r-rm\right)  ^{n-j}%
\genfrac{\{}{\}}{0pt}{}{j+r}{k+r}%
_{r},
\end{equation}
or expressed in terms of the Stirling numbers of the second kind%
\begin{equation}
W_{m,r}\left(  n,k\right)  =%
{\displaystyle\sum\limits_{j=k}^{n}}
\binom{n}{j}m^{j-k}r^{n-j}S\left(  j,k\right)  .\label{dowsti}
\end{equation}

Note that (\ref{rdow}) reduces to the Dowling polynomials by setting $r=1$ and
the $r$-Bell polynomials $B_{r}\left(  n,x\right)  $ by setting $m=1$. In
another recent paper the writer \cite{Rahmani} has shown the relationship of
$r$-Bell numbers to the Bell numbers by%
\[
B_{r}\left(  n,1\right)  =%
{\displaystyle\sum\limits_{k=0}^{n}}
\genfrac{\{}{\}}{0pt}{}{n+r}{k+r}%
_{r}=%
{\displaystyle\sum\limits_{k=0}^{r}}
s\left(  r,k\right)  \phi_{n+k}.
\]

Hence we have%
\begin{equation}
B_{r}\left(  n,x\right)  =%
{\displaystyle\sum\limits_{k=0}^{n}}
\genfrac{\{}{\}}{0pt}{}{n+r}{k+r}%
_{r}x^{k}=%
{\displaystyle\sum\limits_{k=0}^{r}}
s\left(  r,k\right)  x^{-r}\phi_{n+k}\left(  x\right).\label{rbellbell}
\end{equation}

In this section, we show all the results of section $2$ concerning the Dowling
polynomials can be extended to $r$-Dowling polynomials. In particular, the relationship of $r$-Dowling polynomials to the Bell polynomials.
\begin{theorem}
The $r$-Dowling polynomials may be expressed in terms of the Bell polynomials%
\begin{equation}
D_{m,r}\left(  n,x\right)  =%
{\displaystyle\sum\limits_{j=0}^{n}}
\binom{n}{j}m^{j}r^{n-j}\phi_{j}\left(  \frac{x}{m}\right)  .\label{rdowbel}%
\end{equation}

\end{theorem}

\begin{proof}
By using (\ref{rdow}) and (\ref{dowsti}), we get the result.
\end{proof}

Now we want to generalize (\ref{c02}); setting $m=1$ in $\left(  \ref{rdowbel}\right)
,$ we obtain the well-known result (see \cite{mezo1})
\[
B_{r}\left(  n,x\right)  =%
{\displaystyle\sum\limits_{j=0}^{n}}
\binom{n}{j}r^{n-j}\phi_{j}\left(  x\right)  .
\]
It follow from (\ref{rbellbell}) that

\begin{corollary}%
\begin{equation}
{\displaystyle\sum\limits_{k=0}^{n}}
\binom{n}{k}r^{n-k}x^{r}\phi_{k}\left(  x\right)  =%
{\displaystyle\sum\limits_{k=0}^{r}}
s\left(  r,k\right)  \phi_{n+k}\left(  x\right).  \label{man}
\end{equation}

\end{corollary}

\begin{example}
In \cite{mansour}, Mansour and Shattuck defined a sequence $(C_{n})_{n\geq1}$ with
four parameters by means of
\[
C_{n}\left(  a,b,c,d\right)  =abC_{n-1}\left(  a,b,c,d\right)  +cC_{n-1}%
\left(  a+d,b,c,d\right)  ,
\]
where $C_{0}\left(  a,b,c,d\right)  =1$ and they derived some formulas
involving $C_{i}$ and Bell polynomials $\phi_{i}\left(  x\right)  $ defined by%
\[
C_{n}\left(  a,b,c,d\right)  =b^{n}%
{\displaystyle\sum\limits_{j=0}^{n}}
a^{n-j}d^{j}\binom{n}{j}\phi_{j}\left(  \frac{c}{bd}\right)  .
\]

Now, if we will assume that $d$ divides $a$, then we deduce the following
explicit formula%
\[
C_{n}\left(  a,b,c,d\right)  =\frac{\left(  bd\right)  ^{n+a/d}}{c^{a/d}}%
{\displaystyle\sum\limits_{k=0}^{a/d}}
s\left(  a/d,k\right)  \phi_{n+k}\left(  \frac{c}{bd}\right)  ,
\]
by setting $x=c/bd$ and $r=a/d$ in (\ref{man}).

In particular, for $l\geq1$%
\[
l^{n}C_{n}\left(  1,1,\frac{1}{l},\frac{1}{l}\right)  =\frac{1}{l^{n}}%
{\displaystyle\sum\limits_{k=0}^{l}}
s\left(  l,k\right)  \phi_{n+k}.%
\]

\end{example}

\begin{corollary}%
\begin{equation}
m^{n}\phi_{n}\left(  \frac{x}{m}\right)  =%
{\displaystyle\sum\limits_{k=0}^{n}}
\left(  -1\right)  ^{n-k}\binom{n}{k}r^{n-k}D_{m,r}\left(  k,x\right).
\label{beldow}%
\end{equation}

\end{corollary}

To generalize (\ref{c01}), substituting $m=1$ in (\ref{beldow}) and using
(\ref{rbellbell}), we get

\begin{corollary}
The Bell polynomials satisfy the relation%
\begin{equation}
x^{r}\phi_{n}\left(  x\right)  =%
{\displaystyle\sum\limits_{k=0}^{n}}
{\displaystyle\sum\limits_{j=0}^{r}}
\left(  -1\right)  ^{n-k}\binom{n}{k}r^{n-k}s\left(  r,j\right)  \phi
_{k+j}\left(  x\right)  .\label{mout}%
\end{equation}

\end{corollary}

We note that the identity (\ref{mout}) can be viewed as  the inverse Stirling transform
of (see for instance \cite{mihoubi})
\[
\phi_{n+r}\left(  x\right)  =%
{\displaystyle\sum\limits_{k=0}^{n}}
{\displaystyle\sum\limits_{j=0}^{r}}
j^{n-k}S\left(  r,j\right)  \binom{n}{k}x^{j}\phi_{k}\left(  x\right).
\]

Formulas analogous to (\ref{f01}), (\ref{f02}) can be derived. We omit all proofs.

\begin{theorem}
The following results holds true%
\begin{align*}%
{\displaystyle\sum\limits_{k=0}^{n}}
\binom{n}{k}\binom{n+k}{k}\left(  \frac{m}{r}\right)  ^{k}\phi_{k}\left(
\frac{x}{m}\right)    & =%
{\displaystyle\sum\limits_{k=0}^{n}}
\binom{n}{k}\binom{n+k}{k}\frac{1}{r^{k}}D_{m,r}\left(  k,x\right),  \\%
{\displaystyle\sum\limits_{k=0}^{n}}
\binom{n}{k}\binom{n+k}{k}x^{r}\frac{\phi_{k}\left(  x\right)  }{r^{k}}  & =%
{\displaystyle\sum\limits_{k=0}^{n}}
\binom{n}{k}\binom{n+k}{k}\frac{1}{r^{k}}%
{\displaystyle\sum\limits_{j=0}^{r}}
s\left(  r,j\right)  \phi_{k+j}\left(  x\right).  \\
&
\end{align*}

\end{theorem}

\begin{theorem}
The $r$-Dowling polynomials have the Hankel transform%
\[
H_{n}\left(  \frac{D_{m,r}\left(  n,x\right)  }{r^{n}}\right)  =\left(
\frac{mx}{r^{2}}\right)  ^{\binom{n+1}{2}}%
{\displaystyle\prod\limits_{k=1}^{n}}
k!.
\]
\end{theorem}

In particular, $H_{n}\left(  D_{m,r}\left(  n,x\right)  \right)  =H_{n}\left(
D_{m}\left(  n,x\right)  \right)  =\left(  mx\right)  ^{\binom{n+1}{2}}%
{\displaystyle\prod\limits_{k=1}^{n}}
k!.$

\end{document}